\documentclass[a4paper,12pt]{article}

\usepackage{amscd,amssymb,amsmath,amsfonts,amsthm, mathrsfs}
\usepackage{stmaryrd}
\usepackage{graphicx}
\usepackage{verbatim,enumerate,paralist}
\usepackage{textcomp}
\usepackage[colorlinks]{hyperref}
\hypersetup{%
  urlcolor=blue,linkcolor=blue,citecolor=blue
}
\usepackage{pdfpages}

\usepackage[T1]{fontenc}
\usepackage[all,2cell,poly,knot,tips]{xy}
\UseAllTwocells
\CompileMatrices

\def\blue{\color{blue}}

\usepackage[T1]{fontenc}
\usepackage[all,poly,knot,tips]{xy}
\newtheorem{theorem}{Theorem}[section]

\newtheorem{lemma}[theorem]{Lemma}
\newtheorem{proposition}[theorem]{Proposition}
\newtheorem{remark}[theorem]{Remark}
\newtheorem{definition}[theorem]{Definition}
\newtheorem{example}[theorem]{Example}

\newtheoremstyle{step}{2\bigskipamount}{\medskipamount}{\upshape}{}{\itshape}{. }{ }{\underline{Step~\thestep}}
\theoremstyle{step}

\renewcommand{\thestep}{\arabic{step}}

\numberwithin{equation}{section}

\newcommand{\lra}{\longrightarrow}

\newcommand{\Ra}{\Rightarrow}

\newcommand{\ldual}[1]{\mathord{{\let\nolimits\relax\sideset{^\wedge}{}{#1}}}}
\newcommand{\laction}[2]{\mathord{{\let\nolimits\relax\sideset{^{#1}}{}{#2}}}}
\newcommand{\conj}[2]{\mathord{{\let\nolimits\relax\sideset{^{#1}}{}{#2}}}}

\newcommand{\xra}{\xrightarrow}

\def\CA{{\mathscr A}}
\def\CB{{\mathscr B}}
\def\CC{{\mathscr C}}

\def\CR{{\mathscr R}}

\def\CV{{\mathscr V}}
\def\CW{{\mathscr W}}
\def\CX{{\mathscr X}}

\def\dd{{\colon}}
\def\ot{{\otimes}}

\DeclareMathAlphabet{\mathbbe}{U}{bbold}{m}{n}

\def\makebigger#1#2#3{\hbox{#1$\mathsurround=0pt #2{#3}$}}
\def\bigger#1#2{{\relax\mathpalette{\makebigger#1}{#2}}}

\newcommand{\otb}{{\bigger\large{\mathbbe{\otimes}}}}

\begin{document}

\author{George Janelidze\footnote{This author gratefully acknowledges the support of the South African National Research Foundation.}
 \\
{\small{Dept Math. \& Appl. Math., University of Cape Town, Rondebosch 7701,
South Africa}} \\
{\small{<george.janelidze@uct.ac.za>}} \\
Ross Street\footnote{This author gratefully acknowledges the support of Australian Research Council Discovery Grants DP1094883, DP130101969 and DP160101519.}
 \\
{\small{Centre of Australian Category Theory, Macquarie University, NSW 2109, Australia}} \\
{\small{<ross.street@mq.edu.au>}}}

\title{Real sets}
\date{\today}
\maketitle
\begin{center}
{\em Dedicated to Peter Freyd and Bill Lawvere.}
\end{center}

\begin{abstract}
\noindent 
After reviewing a universal characterization of the extended positive real numbers published by Denis Higgs in 1978, 
we define a category which provides an answer to the questions: 
\begin{itemize}
\item what is a set with half an element?
\item what is a set with $\pi$ elements?
\end{itemize}
The category of these extended positive real sets is equipped with a countable tensor product.
We develop somewhat the theory of categories with countable tensors; 
we call the commutative such categories {\em series monoidal} and conclude by only briefly
mentioning the non-commutative possibility called {\em $\omega$-monoidal}.
We include some remarks on sets having cardinalities in $[-\infty,\infty]$.

\end{abstract}

\tableofcontents

\bigskip
\begin{center}
{\em That which is in locomotion must arrive at the half-way stage before it arrives at the goal.} Zeno \cite{Aristotle}
\end{center}

\section{Introduction}\label{intro}

For many years the authors drafted joint notes on a general project dedicated to 
developing the theory of categories with tensor products of infinitely many objects.
As part of that, we were interested in sets with infinite operations.
There is already some literature in this direction: for example, Tarski's
book \cite{Tarski1949}, and the work starting with Linton and Semadeni \cite{Semadeni1973} and leading to a series of papers including Fillmore-Pumpl\"un-R\"ohrl \cite{FPR2002}.

Serendipity led us recently to Higgs' paper \cite{Higgs1978}
which provides a universal property for the set $[0,\infty]$ of extended positive real numbers with structure involving infinite summation.
The paper acknowledges ideas of Huntingdon \cite{Hunt1902} and Tarski \cite{Tarski1949}. 
More importantly for the current Special Volume is Higgs' interesting 
paragraph which begins with the sentence:
\newline \indent{\textit{In conclusion, I would like to say that the stimulus for the introduction
of magnitude modules was a question of Lawvere as to whether a direct
definition of the continuum, appropriate for use in a topos, could be
given.}}
\newline Also, of course, Bill Lawvere \cite{LawMetric} used $[0,\infty]$ as a base for 
recognizing metric spaces as a fertile part of category theory.     

Moreover, as the unary operation of halving is used by Higgs to pin down $[0,\infty]$,
surely there are connections with the work of Peter Freyd \cite{Freyd2008} which involves the mid-point operation. Such relationships, apart from the fact that real intervals are involved, are not yet apparent to the authors.  

Consequently, \cite{Higgs1978} was the trigger for us to focus our infinite tensor 
work on deciding what might be a set with a real cardinality.
The first four sections of the present paper are essentially a reorganization of Higgs' paper, emphasising the structures we later use to provide our categorical version. 

In Section~\ref{smamc}, we define series monoidal categories as categories 
equipped with a countable summation operation appropriately axiomatized. 
Many examples are explained.
What we call Zeno functors in Section~\ref{Zemc} allow us to halve objects; 
these endofunctors universally lead to our category of extended positive real sets. 

The logarithm of a positive real may be negative. Section~\ref{rois} mentions
that and other ideas about capturing all real numbers and sets.

One of the purposes of symmetric monoidal categories is to serve as bases for enriched categories. In Section~\ref{ceismc}, we look at categories enriched in a series
monoidal category and show that they form a series monoidal 2-category.
On the excuse that one of our constructions could lead us to a non-symmetric
example, we briefly look in the last Sections~\ref{omegacase} and \ref{omegamoncats} at
non-symmetric infinitary operations.

We suspect the reason no one has suggested our
construction of the category of positive real sets is that the Higgs paper was looked
at more for its contribution to measure theory \cite{Higgs1980} and that 
categories with infinite tensor products have not had much attention. 
 
\bigskip
\noindent \textbf{Note:} In Section~\ref{smamc} we explain that Examples~\ref{inftenRmod}, \ref{inftenRsermod}, \ref{Rmodadj}, and \ref{Rsermodadj} in the published version 
[\textit{Real sets}, Tbilisi Math. J. \textbf{10(3)} (2017) 23--49] were incorrect.    

\section{Series magmas and series monoids}\label{sermagmon}

Let $\mathbb{N}$ denote the natural numbers which include $0$.
For sets $X$ and $A$, we write $A^X$ for the set of functions $a\dd X\to A$ 
and we often put $a_x\dd =a(x)$ and $(a_x)_{x\in X}\dd =a$. 
Given $0\in A$, define
\begin{eqnarray*}
\delta\dd A\lra A^{\mathbb{N}\times \mathbb{N}}
\end{eqnarray*}
by   
 \begin{equation*}
\delta(a)_{m,n} =
\begin{cases}
a & \text{if } m = n, \\
0 & \text{if } m\neq n \ .
\end{cases}
\end{equation*}
We identify $\delta$ with its composite with either of the canonical
isomorphisms $\sigma_1, \sigma_2\dd A^{\mathbb{N}\times \mathbb{N}}\cong (A^{\mathbb{N}})^{\mathbb{N}}$, where
$$\sigma_1(a)(m)(n)=a_{m,n}=\sigma_2(a)(n)(m) \ ,$$ 
since $\sigma_1\circ\delta=\sigma_2\circ\delta$. 
We also write $\delta_n\dd A\to A^{\mathbb{N}}$ for $a\mapsto (\delta(a)_{m,n})_{m\in \mathbb{N}}$. 

\begin{definition} {\em A {\em series magma} is a set $A$ equipped with
an element $0\in A$ and a function 
\begin{eqnarray*}
\Sigma : A^{\mathbb{N}} \lra A \ , \ (a_i)_{i\in \mathbb{N}} \longmapsto \Sigma_{i\in \mathbb{N}} a_i
\end{eqnarray*}
such that the following diagram commutes for all $n\in \mathbb{N}$.
\begin{equation}\label{zerocond}
\begin{aligned}
\xymatrix{
A \ar[rd]_{A^{!}}\ar[rr]^{\delta \phantom{aa}}   & & (A^{\mathbb{N}})^{\mathbb{N}} \ar[ld]^{\Sigma^{\mathbb{N}}} \\
& A^{\mathbb{N}}  &
}
\end{aligned}
\end{equation}} 
\end{definition}

For any series magma $A$ and subset $S\subseteq \mathbb{N}$, we can define an operation $\Sigma_S\dd A^S\to A$ whose value at $a\in A^S$ is 
\begin{eqnarray}\label{sum/S}
\Sigma_{n\in S}a_n=\Sigma_{n\in \mathbb{N}}c_n
\end{eqnarray}
where
 \begin{equation*}
c_n =
\begin{cases}
a_n & \text{if }  n\in S\\
0 & \text{otherwise} \ .
\end{cases}
\end{equation*}

Since series magmas are models of an algebraic theory, there is a
corresponding notion of morphism, that is, a function $f\dd A\to B$
such that $f(0)=0$ and the following square commutes.
\begin{equation}\label{smagmorph}
\begin{aligned}
\xymatrix{
A^{\mathbb{N}} \ar[rr]^-{\Sigma} \ar[d]_-{f^{\mathbb{N}}} && A \ar[d]^-{f} \\
B^{\mathbb{N}} \ar[rr]_-{\Sigma} && B}
\end{aligned}
\end{equation}
Also, the resultant category $\mathrm{SerMg}$ of series magmas is both complete and cocomplete, and is Barr-Tierney exact.
The forgetful functor $\mathrm{U}\dd \mathrm{SerMg}\to \mathrm{Set}$ 
is monadic.
The monad generated by $\mathrm{U}$ and its left adjoint preserves
$\aleph_1$-filtered colimits. 

An aspect of all this is that $A^{\mathbb{N}}$ is the underlying set for the cotensor of the set $\mathbb{N}$ with the series magma $A$; the series magma
structure consists of the constant sequence $0=\delta(0)_0$ 
and $\Sigma^{\mathbb{N}}\dd (A^{\mathbb{N}})^{\mathbb{N}}\to A^{\mathbb{N}}$. 

Here is an easy Eckmann-Hilton-type result.

\begin{proposition}
Suppose a set $A$ has two series magma structures $\Sigma$ and $\Sigma'$ with the same $0$.
If $\Sigma'\dd A^{\mathbb{N}}\to A$ is a morphism for the $\Sigma$
structure on $A$ then $\Sigma'=\Sigma$ and, for all $a\in A^{\mathbb{N}\times \mathbb{N}}$, 
\begin{eqnarray}\label{sumswap}
\Sigma_{m\in \mathbb{N}}\Sigma_{n\in \mathbb{N}}a_{m,n}= \Sigma_{n\in \mathbb{N}}\Sigma_{m\in \mathbb{N}}a_{m,n} \ .
\end{eqnarray}
\end{proposition}
\begin{proof}
The morphism condition \eqref{smagmorph} for $\Sigma'$ is
\begin{eqnarray*}
\Sigma'_{m\in \mathbb{N}}\Sigma_{n\in \mathbb{N}}a_{m,n}= \Sigma_{n\in \mathbb{N}}\Sigma'_{m\in \mathbb{N}}a_{m,n} \ .
\end{eqnarray*}
In this, for any $b\in A^{\mathbb{N}}$,  take the diagonal matrix 
$a_{m,n}= \delta(b_m)_{m,n}$.
Using \eqref{zerocond} for both sums, we obtain $\Sigma'_{m\in \mathbb{N}}b_m = \Sigma_{n\in \mathbb{N}}b_n$;
that is, $\Sigma'(b)=\Sigma(b)$. 
\end{proof}

\begin{definition} {\em A {\em series monoid} is a series magma satisfying \eqref{sumswap}. Write $\mathrm{SerMn}$ for the full subcategory of $\mathrm{SerMg}$ consisting of the series monoids.} 
\end{definition}

\begin{example}\label{extnat}
{\em The natural numbers $\mathbb{N}\cup \{\infty\}$, 
extended to include $\infty$, 
is a series monoid with $0$ the natural number $0$ and
 \begin{equation*}
\Sigma_{n\in \mathbb{N}} a_n =
\begin{cases}
\sum_{n=0}^{\infty}{a_n} & \text{if } a \text{ has finite support} \\
\infty & \text{otherwise} \ .
\end{cases}
\end{equation*}}
\end{example} 

\begin{example}\label{exreal}
{\em Similarly, the non-negative real numbers $[0,\infty]$, extended to include $\infty$, 
is a series monoid with $0$ the real number $0$ and
 \begin{equation*}
\Sigma_{n\in \mathbb{N}} a_n =
\begin{cases}
\sum_{n=0}^{\infty}{a_n} & \text{if the series converges}  \\
\infty & \text{otherwise} \ .
\end{cases}
\end{equation*}}
\end{example} 

\begin{example}\label{pointwise}
{\em For any series monoid $A$ and any set $X$, there is the pointwise
series monoid structure on $A^X$.
For various choices of $X$ and $A$, there can be interesting series submonoids of $A^X$. 
With $X$ a measurable space and $A=[0,\infty]$, the measurable functions
$f\dd X\to [0,\infty]$ form a series submonoid of $[0,\infty]^X$. 
With $X=A=[0,\infty]$, the continuous non-decreasing functions
form a series submonoid of $[0,\infty]^{[0,\infty]}$. 
}
\end{example} 

\begin{example}\label{exsuplattice}
{\em Any partially ordered set $A$ admitting countable suprema is a 
series monoid with $0$ the bottom element and $\Sigma$ equal to
the countable supremum operation $\bigvee$.
}
\end{example} 

 \begin{proposition}\label{commutative}
 Suppose $A$ is a series monoid and 
 $\xi \dd \mathbb{N} \to \mathbb{N}$ is an injective function. 
 If $a\in A^{\mathbb{N}}$ is such that $a_n=0$ for $n$ not in the image of $\xi$,
 then $\Sigma_na_{\xi(n)}=\Sigma_na_n$.
 \end{proposition}
 \begin{proof}
 We define $b_{m,n}$ to be $a_{\xi(m)}$
 for $n=\xi(m)$ and to be $0$ otherwise. 
 Since $\xi$ is injective, each row and column of the matrix $b$
 has at most one non-zero entry. Applying \eqref{sumswap} to $b$
 and using \eqref{zerocond}, we obtain the result. 
 \end{proof}

 \begin{remark}\label{diag}
 {\em Similarly, if $\xi \dd \mathbb{N} \to \mathbb{N}\times \mathbb{N}$ is an injective function and $a_{m,n}=0$ for $(m,n)$ not in the image of $\xi$,
 then $\Sigma_na_{\xi(n)}=\Sigma_{(m,n)}a_{(m,n)}$ where, of course,
 the right-hand side is either side of \eqref{sumswap}. 
 We leave this as an exercise.} 
 \end{remark} 
 
As a particular case of \eqref{sum/S}, we can define a binary operation $a_1+a_2 = \Sigma_{n\in \{1,2\}}a_n$.
This makes the series monoid $A$ into a commutative monoid with $0$ as identity for $+$.
Moreover, $\Sigma \dd A^{\mathbb{N}} \to A$ is a monoid morphism.
The informal notation
\begin{equation*}
\Sigma_na_n = a_0+a_1+a_2+ \dots
\end{equation*}
can be suggestive.

We can also make $A$ into a pre-ordered set by defining $a\le b$ when there exists $u$ with $a+u=b$. It is clearly reflexive, transitive, has $0$ as
least element, and is respected by $\Sigma$.

 \begin{definition} {\em A series monoid is called {\em idempotent} when, for all $c\in A$ and $a\in A^{\mathbb{N}}$ such that $a_n\neq 0$ implies $a_n=c$, it follows that $\Sigma_na_n = c$ holds.} 
\end{definition}

\begin{proposition}\label{charsuplattice}
 A series monoid arises from a partially ordered set as in Example~\ref{exsuplattice} if and only if it is idempotent.
  \end{proposition}
 \begin{proof}
Suppose $A$ is an idempotent series monoid.
We can prove that the order is antisymmetric. 
For, take $a\le b$ and $b\le a$; so we have $a +u = b$ and $b+v=a$. 
Then $a = a+u+v = a + (u+v)+(u+v)+(u+v)+\dots$ by idempotence.
So $$a = a + u + (v+u)+(v+u)+\dots = a + (u+v)+(u+v)+\dots +u = a +u=b \ .$$
To see that $\Sigma_na_n$ is the supremum of $\{a_n:n\in \mathbb{N}\}$, we
have  $a_m+\Sigma_na_n = \Sigma_na_n$ by commutativity and idempotency; so $a_m \le \Sigma_na_n$. Now suppose $a_n\le c$ for
all $n$. This means there exist $u_n$ with $a_n+u_n=c$ for all $n$.
Since $\Sigma$ is a monoid morphism and because of idempotency,
we have $\Sigma_na_n+\Sigma_nu_n = c$. So $\Sigma_na_n\le c$. 
 \end{proof}
 
 The forgetful functor $\mathrm{U}\dd \mathrm{SerMn}\to \mathrm{Set}$
 has a left adjoint whose value at $1$ can be made explicit.  

\begin{proposition}\label{freesmon1}
 The free series monoid on a single generating element is $\mathbb{N}\cup \{\infty\}$ as in Example~\ref{extnat}.
 In other words, $\mathbb{N}\cup \{\infty\}$ is a representing object for the functor $\mathrm{U}$.
  \end{proposition}
 \begin{proof}
Given a series monoid $A$, we will show that 
$$\mathrm{ev}_1\dd \mathrm{SerMn}(\mathbb{N}\cup \{\infty\}, A)\to A$$
taking $f$ to $f(1)$ is bijective. 
Take $a\in A$ and define $f_a\dd \mathbb{N}\cup \{\infty\}\to A$ by
$$f_a(n)=na=\overbrace{a+\dots +a}^n$$
and $f_a(\infty)=a+a + \dots$.
Then $f_a(1)=a$. Also, for $f_a\dd \mathbb{N}\cup \{\infty\}\to A$, we have
$f_{f(1)}(n) = f(1)+\dots f(1) = f(1+\dots +1) = f(n)$ and $f_{f(1)}(\infty) = f(1)+ f(1)\dots  = f(1+1+\dots) = f(\infty)$, so $f_{f(1)}=f$.     
 \end{proof}
 
 Countable products and sums (= coproducts) in $\mathrm{SerMn}$ are special: they coincide. 
 We shall explain this although it is much like the case of finite direct products for commutative monoids.

Consider a sequence $(A_k)_{k\in \mathbb{N}}$ of series monoids. 
The cartesian product $\prod_{k\in \mathbb{N}} A_k$ becomes a series monoid by defining $\Sigma$ to be the composite
$$(\prod_{k\in \mathbb{N}} A_k)^{\mathbb{N}}\cong \prod_{k\in \mathbb{N}} (A_k)^{\mathbb{N}} \xra{\prod_{k\in \mathbb{N}}\Sigma} \prod_{k\in \mathbb{N}} A_k \ .$$
The projections $\mathrm{pr}_k : \prod_{k\in \mathbb{N}} A_k \lra A_k$ are all morphisms of  series monoids. 
This gives the product in the category $\mathrm{SerMn}$.  

Now, we can define morphisms $\mathrm{in}_k : A_k \lra \prod_{h\in \mathbb{N}} A_h$ by
\begin{equation*}
\mathrm{in}_k(a)_h =
\begin{cases}
a & \text{for } h = k , \\
0 & \text{for } h \ne k.
\end{cases} 
\end{equation*}

\begin{proposition}\label{countabledirectsum}
The family of morphisms $\mathrm{in}_k : A_k \lra \prod_{h\in \mathbb{N}} A_h$, for $k\in \mathbb{N}$, is a coproduct in the category $\mathrm{SerMn}$. The following formulas hold:
$$\Sigma_{k\in \mathbb{N}} \mathrm{in}_k \circ \mathrm{pr}_k = 1_{\prod_{h\in \mathbb{N}} A_h} \ ,$$
$$ \mathrm{pr}_k \circ \mathrm{in}_m =
\begin{cases} 
1_{A_k} & \text{for } k=m , \\
0 & \text{for } k \ne m . 
\end{cases} $$
\end{proposition}
\begin{proof}
The second sentence is an immediate consequence of the definitions.
To prove we have a coproduct, take a family of morphisms 
$f_k : A_k \lra B$ into a series monoid $B$. 
Using the formulas of the second sentence, we deduce that the only morphism 
$f : \prod_{k\in \mathbb{N}} A_k \lra B$ with $f \circ \mathrm{in}_k = f_k$ for all $k \in \mathbb{N}$ is 
$f = \Sigma_{k\in \mathbb{N}} f_k$.    
\end{proof}

For families $(A_i)_{i\in I}$ with $I$ not countable, the product is still the cartesian product with
pointwise operations. The coproduct is the subobject consisting of the families of countable support.
With this, it follows from Proposition~\ref{freesmon1} that we can describe all free series monoids
(since free functors preserve coproducts and every set is a coproduct of one-element sets $1$).

\begin{proposition}\label{freesmonall}
 The free series monoid on a set $X$ is the subobject of $(\mathbb{N}\cup\{\infty\})^X$ (as in Example~\ref{extnat})
 consisting of the functions of countable support.
  \end{proposition}

\section{The symmetric closed structure}\label{scs}

For series monoids $A$ and $B$, we write $\mathrm{ser}(A,B)$ for the set $\mathrm{SerMn}(A,B)$
equipped with the pointwise series monoid structure.
From Proposition~\ref{freesmon1}, we have an isomorphism 
$$\mathrm{i}_{\mathrm{ser}(A,A)} : [\mathbb{N}\cup \{\infty\},\mathrm{ser}(A,A)] \cong \mathrm{ser}(A,A) \ ,$$ 
and so a morphism $$\mathrm{j}_A : \mathbb{N}\cup \{\infty\} \lra \mathrm{ser}(A,A)$$ 
corresponding to the identity morphism $1_A\in \mathrm{ser}(A,A)$.  

Since $\Sigma$ for each $\mathrm{ser}(C,D)$ is defined pointwise in $C$, we have a morphism
$$\mathrm{L}^A : \mathrm{ser}(B,C) \lra \mathrm{ser}(\mathrm{ser}(A,B),\mathrm{ser}(A,C))$$
defined by $\mathrm{L}^A(g)(f) = g\circ f$.

There is also an isomorphism 
$$\mathrm{s}_{ABC} : \mathrm{ser}(A,\mathrm{ser}(B,C))\cong \mathrm{ser}(B,\mathrm{ser}(A,C))$$ 
defined by noting that both sides are isomorphic to the pointwise series monoid of functions 
$f : A\times B \lra C$ for which all 
$f(a,-):B\lra C$ and $f(-,b):A\lra C$ are morphisms.   

See \cite{EilKel1966} and \cite{scc} for the definition of closed category and the definition of category enriched in a closed category. 

\begin{proposition}\label{closed structure}
 A symmetric closed structure on the category $\mathrm{SerMn}$ is defined by 
 $(\mathrm{i}, \mathrm{j}, \mathrm{L}, \mathrm{s})$. 
 The obvious inclusions $1\lra \mathbb{N}\cup \{\infty\}$ and $\mathrm{U}[A,B] \lra (\mathrm{U}B)^{\mathrm{U}A}$ provide the forgetful functor $\mathrm{U} : \mathrm{SerMn} \lra \mathrm{Set}$ with a closed structure.    
\end{proposition}
\begin{proof}
To check that the axioms pass from those axioms for the cartesian closed structure on $\mathrm{Set}$ we use the facts that each $\mathrm{U}[A,B] \lra (\mathrm{U}B)^{\mathrm{U}A}$ is a monomorphism, and that $\mathbb{N}\cup \{\infty\}$ is free on $1$ (Proposition~\ref{freesmon1}).     
\end{proof}

\begin{proposition}
The forgetful functor
$\mathrm{U}\dd \mathrm{SerMn} \lra \mathrm{Set}$
is monadic of rank $\aleph_1$.
The left adjoint is defined on objects in Proposition~\ref{freesmonall}. 
The monad on $\mathrm{Set}$ generated by the adjunction is closed (= monoidal).
\end{proposition}   
\begin{proof} The theory of series monoids is commutative.
\end{proof}

By the general theory provided by Kock \cite{Kock1971}, the closed structure of Section~\ref{scs} (see Proposition~\ref{closed structure}) is monoidal. We will write $A\otimes B$ for the tensor product of series monoids. 
We are interested in monoids for this tensor product; they might be called {\em series rigs}.
(The term ``rig'' was used by Lawvere and Schanuel; the lack of an ``n'' in the word was to indicate
the lack of negatives in the otherwise ring.)

Let $A$ be a commutative series rig; 
that is a commutative monoid in the symmetric monoidal category $\mathrm{SerMn}$.
We will write the operation of the monoid multiplicatively. This product
distributes over $\Sigma$, and $0$ acts as a zero. 
By associativity and commutativity, for each family $a=(a_n)_{n\in S}$ of elements of $A$
indexed by a finite set $S$, there is an element $\Pi_{m\in S}a_m \in A$. 

Write $\binom{\mathbb{N}}{n}$ for the set of subsets of $\mathbb{N}$ of cardinality $n$.

Now for $a\in A^{\mathbb{N}}$, define
\begin{eqnarray}\label{fP}
\mathrm{P}a= \mathrm{P}_{r\in \mathbb{N}}a_r = \Sigma_{0<n\in \mathbb{N}}\Sigma_{S\in \binom{\mathbb{N}}{n}}\Pi_{m\in S}a_m \ .
\end{eqnarray}
 Less formally,
 \begin{eqnarray}\label{lfP}
\mathrm{P}a= \Sigma_ia_i + \Sigma_{i<j}a_ia_j + \Sigma_{i<j<k}a_ia_ja_k + \dots \ .
\end{eqnarray}
In particular,
\begin{eqnarray}\label{lfP2}
\mathrm{P}(a_0,a_1,0,0,\dots) = a_0 + a_1 + a_0a_1 \ .
\end{eqnarray}
 \begin{proposition}\label{Pprop}
 Any commutative monoid $A$ in the monoidal category $\mathrm{SerMn}$ has a series monoid
 structure defined by $0\in A$ and $\mathrm{P}\dd A^{\mathbb{N}}\to A$.   
 \end{proposition}
 
  \begin{remark}\label{previousexamples}
{\em Notice that each of Examples~\ref{extnat} and \ref{exreal} can be obtained using 
Proposition~\ref{Pprop} from an example of the countable supremum type of Example~\ref{exsuplattice}.
For Example~\ref{extnat}, take the sup-lattice $\mathbb{N}\cup \{\infty\}$ with addition as monoid structure.
For Example~\ref{exreal}, take the sup-lattice $[0,\infty]$ with addition as monoid structure.
Indeed, Example~\ref{exsuplattice} is obtained from itself using the monoid structure of finite sup.}   
 \end{remark}
 
 \begin{remark}\label{foreshadomega}
{\em Notice that the unit for the monoid $A$ is not needed for Proposition~\ref{Pprop}.
The formula \eqref{fP} does not require commutativity of $A$ but then we only obtain a ``non-commutative series monoid'' in a sense to be pursued in Section~\ref{omegacase}.}   
 \end{remark}
 
 \begin{remark}\label{DayRemark}
{\em As pointed out by Day \cite{Day55}, the ordered set $[0,\infty]$ is 
$*$-autonomous with multiplication as tensor product and dualizing object
the same as the tensor unit $1$, internally homming into which gives
reciprocal as the equivalence
\begin{eqnarray*}
S\dd [0,\infty]^{\mathrm{op}}\lra [0,\infty] \ . 
\end{eqnarray*}
 In fact we see that $S(\alpha)= \frac{1}{\alpha}$
 is actually the dual of each $0<\alpha <\infty$,
 while $S(0)=\infty$ and $S(\infty)=0$. 
 Day further points out that the natural
logarithm gives an inverse to a monoidal equivalence
\begin{eqnarray*}
\mathrm{exp}\dd [-\infty,\infty]\lra [0,\infty] 
\end{eqnarray*}
where the tensor product in the domain is addition, 
and therefore is $*$-autonomous.}   
 \end{remark} 
 
 Motivated by Remark~\ref{DayRemark}, we take our commutative monoid $A$
 in $\mathrm{SerMn}$ and create another copy of the set $A$ which we
 will denote by $\ell A$. The elements of $\ell A$ will be denoted by $\ell a$
 where $a\in A$. We make $\ell A$ into a commutative monoid by defining
 \begin{eqnarray}\label{plusell}
\ell a + \ell b = \ell(ab) \ \text{, } \ 0 = \ell 1 \ \text{ and } \ -\infty = \ell0.
\end{eqnarray}

By definition, if $a\le 1$ in $A$ then there exists $u\in A$ with $a+u=1$. 
We can form the geometric series $v = 1 + u + u^2 + u^3 + \dots$ in $A$;
then $av+ uv = (a+u)v = v = 1 + u (1 + u + u^2 + \dots) = 1 + uv$.
If $uv$ can be cancelled, then $v=a^{-1}$. Then we have
 \begin{eqnarray}\label{minusell}
\ell (a^{-1})  = -\ell a \ .
\end{eqnarray}
We also have some countable sums in $\ell A$: 
 \begin{eqnarray}\label{inftyplusell}
\Sigma_n\ell (1+u_n)  = \Pi_n(1+u_n) = 1 + \mathrm{P}_nu_n \ .
\end{eqnarray}
When $1$ is cancellative in the additive monoid $A$, then $1\le a$
implies $1+u = a$ for a unique $u$; in this case, \eqref{inftyplusell}
defines a sum for sequences of ``non-negative elements''
$\ell a_n = \ell (1+u_n)$ in $\ell A$.  

\section{Zeno morphisms and magnitude modules}

Given an endomorphism $f\dd A\to A$ in $\mathrm{SerMn}$, define
$\tilde{f}\dd A\to A$ by the geometric series
\begin{eqnarray}
\tilde{f} = \Sigma_{n\in \mathbb{N}}f^{\circ (n+1)}
\end{eqnarray}
in the pointwise structure on $A^A$. This $\tilde{f}\dd A\to A$ satisfies
\begin{eqnarray}\label{tildeqn}
f\circ (1_A+\tilde{f}) =  f+f\circ \tilde{f} = \tilde{f}
\end{eqnarray}
and is a morphism in $\mathrm{SerMn}$.
 
 \begin{definition} {\em A {\em Zeno morphism} in $\mathrm{SerMn}$ is an endomorphism $h\dd A\to A$ such that $\tilde{h}=1_A$.
 A {\em magnitude module} in the sense of Higgs \cite{Higgs1978} is a
 series monoid equipped with a Zeno morphism $h$.} 
\end{definition}

Magnitude modules are models of an algebraic theory; they are series monoids with an extra unary
operation satisfying one extra axiom. 

From \eqref{tildeqn}, any Zeno morphism satisfies
\begin{eqnarray}\label{half}
h + h = 1_A 
\end{eqnarray}
and so can be regarded as the operation of halving.

\begin{example}\label{exnathalf}
{\em When $A=\mathbb{N}\cup \{\infty\}$ as in Example~\ref{extnat}, there exists no Zeno morphism since \eqref{half} gives the contradiction $h(1)+h(1) = 1$.}
\end{example}

\begin{example}\label{exrealhalf}
{\em When $A=[0,\infty]$ as in Example~\ref{exreal}, the unique Zeno morphism
is defined by $h(a)=\frac{1}{2}a$ and $h(\infty)=\infty$.}
\end{example}

\begin{example}\label{measurable}
{\em Refer back to Example~\ref{pointwise} for any set $X$ and any magnitude module $A$.
The pointwise Zeno morphism makes both $A^X$ and $[X,A]$ into magnitude modules.  
For $X$ a measurable space, 
Higgs \cite{Higgs1978} observed that the measurable functions
$f\dd X\to [0,\infty]$ form a magnitude submodule of $[0,\infty]^X$ and
that magnitude module morphisms from there into $[0,\infty]$ are the
countably-additive $[0,\infty]$-valued integrals on $X$. }
\end{example}

\begin{example}\label{exsuplatticehalf}
{\em For a partially ordered set $A$ as in Example~\ref{exsuplattice}, the identity function $h(a)=a$ is Zeno. 
}
\end{example} 

\begin{theorem}[Higgs] 
The free magnitude module on a single generating element is $[0,\infty]$ as in Example~\ref{exrealhalf}. 
\end{theorem}
\begin{proof}
The proof is given in Section 4 of \cite{Higgs1978} so we shall only give an indication.
Every natural number is a finite sum $1+1+\dots +1$ while $\infty=1+1+\dots$.
Every positive real is the sum of a natural number and a real number $t$ 
in the interval $[0,1)$. However, we have the binary expansion
$t=\Sigma_{n\in \mathbb{N}}h^{\circ {m_n}}(1)$, where $m_0, m_1, \dots$
is a sequence of strictly positive integers.
Consequently, $1\in [0,\infty]$ generates. 
The fact that we have equality of binary expansions such as $1.000\dots = 0.111\dots$ is no
problem since $1=\tilde{h}(1)$.
\end{proof}

\begin{remark}\label{explicitconstruct} 
{\em The construct of the extended reals $[0,\infty]$ as a quotient of a free series monoid is as follows.
Let $\mathbb{M}= (\mathbb{N}\cup \{\infty\})^{\mathbb{N}}$ be the free series monoid on 
$\mathbb{N}$ (Proposition~\ref{freesmonall}). We have the universal morphism
$\chi \dd \mathbb{N}\to \mathbb{M}$ taking $n$ to the function $\chi_n$ which has only
non-zero value at $n$ and that value is $1$.
Take the smallest series monoid congruence $\sim$ including the relations
$$\chi_n \sim \chi_{n+1} + \chi_{n+2} + \chi_{n+3} + \dots \ .$$
A consequence of these relations is $\chi_n \sim 2 \chi_{n+1}$.
Now $\mathbb{M}/_{\sim}$ becomes a magnitude module by the Zeno function $h$
defined by $h(\chi_n)= \chi_{n+1}$, that is, $h$ is induced by successor on $\mathbb{N}$.
We have the magnitude module isomorphism 
$$\mathbb{M}/_{\sim} \ \xra{\cong} [0,\infty] \ \text{, } \ \chi_n \mapsto \frac{1}{2^n} \ .$$  
Incidentally, another way of constructing the reals from endomorphisms of $\mathbb{N}$ is explained in \cite{EffR}.
The question of how to define multiplication for any construction of the real number system is always of interest.
In \cite{EffR}, it is simply induced by composition of functions.
For decimals, it is tricky. We now turn to the multiplication in our context.}   
\end{remark}

For any series monoid $A$, by freeness there is a magnitude module morphism
$$\lambda\dd [0,\infty] \lra \mathrm{ser}(A,A)$$  
taking the generator $1\in [0,\infty]$ to the identity function of $A$;
see Section~\ref{scs}.
This gives an action 
$$\centerdot \ \dd [0,\infty] \otimes A \lra A$$
of $[0,\infty]$ on $A$ defined by $\alpha \centerdot a = \lambda(\alpha)(a)$.  
In particular, we have a monoid structure
$$\centerdot \ \dd [0,\infty]\otimes [0,\infty] \lra [0,\infty]$$
on $[0,\infty]$ in the monoidal category $\mathrm{SerMn}$; 
the unit is the generator $1$ of $[0,\infty]$.
This gives a monad $[0,\infty]\otimes -$ on $\mathrm{SerMn}$.

\begin{proposition}[Higgs]\label{EMmagmod} 
The Eilenberg-Moore algebras for the monad \\ $[0,\infty]\otimes-$ 
on $\mathrm{SerMn}$ are precisely the magnitude modules.
\end{proposition}

The ``magnitude'' terminology comes from Huntingdon \cite{Hunt1902}
who took magnitudes in the unextended strictly positive reals $(0,\infty)$.
  
\begin{remark}\label{paradoxicalreals}
{\em There is also what one might call the {\em paradoxical positive reals} where the
geometric series $\frac{1}{2}+\frac{1}{4}+\frac{1}{8}+\dots $ is not $1$.
It is an example of a very general simple construction of an additive monoid structure,
on the disjoint union $\{0\}+X+S$, from any semigroup morphism $k\dd S\to X$ in $\mathrm{Set}$.
It freely adds the zero element $0$ to the semigroup $X+S$ whose addition $\mu_{X+S}$ is the composite
\begin{eqnarray*}
(X+S)\times (X+S)\xra{\cong}X\times X+X\times S+S\times X+S\times S\xra{1+1\times k+k\times 1+1} \\ 
X\times X+X\times X+X\times X+S\times S \xra{[\mu_X,\mu_X,\mu_X,\mu_S]}X+S \ .
\end{eqnarray*}
For our particular example, let $X$ be the additive semigroup of positive real numbers, presented as infinite binary expansions 
excluding those having only finitely many terms equal to $1$. 
Let $S$ be the set of all positive rational numbers of the form $\frac{m}{2^n}$, 
where n and m are integers, presented as binary expansions $s$ having only finitely many terms equal to $1$. 
Define $k\dd S\to X$ to replace the last $1$ in $s$ with a $0$ and all the later $0$s by $1$s;
for example, $k(1.00\dots)=0.11\dots$. 
Then $\{0\}+X+S$ is our monoid $\mathrm{ZP}[0,\infty)$ of paradoxical positive reals.
In there we have, for example, $$(0.11\dots) + (0.11\dots)  = (0.11\dots) +  (1.00\dots) = 1.11 \dots$$
and 
$$(1.00 \dots) + (1.00 \dots) = 10.00 \dots \ .$$ 
We note that $\mathrm{ZP}[0,\infty)$ is not only an ordered monoid, but, considered as an ordered set with $\infty$ added, is the free completion of $S$ under arbitrary joins; 
this is in contrast to the ordinary $[0,\infty]$, which is the existing-join-preserving completion.  
}
\end{remark}

\section{Series monoidal categories}\label{smamc}

Let $\CA$ be a category. 
Given an object $0$ of $\CA$, define the functor
\begin{eqnarray*}
\delta\dd \CA\lra \CA^{\mathbb{N}\times \mathbb{N}}
\end{eqnarray*}
by   
 \begin{equation*}
\delta(A)_{m,n} =
\begin{cases}
A & \text{if } m = n, \\
0 & \text{if } m\neq n.
\end{cases}
\end{equation*}
For $A\in \CA^{\mathbb{N}}$ and a functor $\Sigma\dd \CA^{\mathbb{N}}\to \CA$, note that $\delta(\Sigma_nA_n)$ 
and $\Sigma_n\delta(A_n)$ are not too different:  
 \begin{equation*}
\delta(\Sigma_nA_n)_{r,s} =
\begin{cases}
\Sigma_nA_n & \text{if } r = s, \\
0 & \text{if } r\neq s
\end{cases}
\end{equation*}
while
 \begin{equation*}
\Sigma_n\delta(A_n)_{r,s} =
\begin{cases}
\Sigma_nA_n & \text{if } r = s, \\
\Sigma(0,0,\dots) & \text{if } r\neq s \ .
\end{cases}
\end{equation*}
So, if we have an isomorphism $\lambda_00\dd 0\to \Sigma(0,0,\dots)$, then there is an induced isomorphism
$\overline{\lambda_00}\dd \delta(\Sigma_nA_n) \to \Sigma_n\delta(A_n)$ 
which is the identity on the diagonal and $\lambda_00$ elsewhere.  

\begin{definition} {\em A {\em series monoidal category} is a category 
$\CA$ equipped with an object $0\in \CA$, a functor 
\begin{eqnarray*}
\Sigma : \CA^{\mathbb{N}} \lra \CA \ , \ (A_i)_{i\in \mathbb{N}} \longmapsto \Sigma_{i\in \mathbb{N}} A_i
\end{eqnarray*}
and natural isomorphisms 
\begin{equation}\label{smoncatdata}
\begin{aligned}
\xymatrix{
\CA^{\mathbb{N}\times \mathbb{N}} \ar[d]_{\sigma_1}^(0.8){\phantom{aaaaaa}}="1" \ar[rr]^{\sigma_2}  && (\CA^{\mathbb{N}})^{\mathbb{N}} \ar[d]^{\Sigma^{\mathbb{N}}}_(0.8){\phantom{aaaaaa}}="2" \ar@{=>}"1";"2"^-{\gamma}
\\
(\CA^{\mathbb{N}})^{\mathbb{N}} \ar[d]_-{\Sigma^{\mathbb{N}}} && \CA^{\mathbb{N}} \ar[d]^-{\Sigma} & ,
\\
\CA^{\mathbb{N}} \ar[rr]_{\Sigma} & & \CA
}
\quad
\xymatrix{
\CA \ar[rd]_{\CA^!}^(0.5){\phantom{aaa}}="1" \ar[rr]^{\delta}  && (\CA^{\mathbb{N}})^{\mathbb{N}} \ar[ld]^{\Sigma^{\mathbb{N}}}_(0.5){\phantom{aaa}}="2" \ar@{=>}"1";"2"^-{\lambda}
\\
& \CA^{\mathbb{N}} \ &  
}
\end{aligned}
\end{equation}
subject to the conditions that the components
of the $\lambda_n$ at $0$ are all equal and diagrams \eqref{gammacond} 
and \eqref{lambdacond} commute.} 
\end{definition}
\begin{equation}\label{gammacond}
\begin{aligned}
\xymatrix{
& \Sigma_m\Sigma_p\Sigma_nA_{mnp} \ar[rd]^-{\phantom{A}\gamma\Sigma_n A_{\centerdot n\centerdot}}  & \\
\Sigma_m\Sigma_n\Sigma_pA_{mnp} \ar[ru]^-{\Sigma_m\gamma A_{m\centerdot\centerdot}\phantom{A}} \ar[d]_-{\gamma\Sigma_p A_{\centerdot\centerdot p}} & & \Sigma_p\Sigma_m\Sigma_nA_{mnp} \ar[d]^-{\Sigma_p\gamma A_{\centerdot\centerdot p}} \\
\Sigma_n\Sigma_m\Sigma_pA_{mnp} \ar[rd]_-{\Sigma_n\gamma A_{\centerdot n\centerdot}} & & \Sigma_p\Sigma_n\Sigma_mA_{mnp}  \\
& \Sigma_n\Sigma_p\Sigma_mA_{mnp}  \ar[ru]_-{\gamma\Sigma_m A_{m\centerdot\centerdot}} 
}
\end{aligned}
\end{equation}
\begin{equation}\label{lambdacond}
\begin{aligned}
\xymatrix{
\Sigma_m\Sigma_n\delta(A_m)_{n,p} \ar[rr]^-{\gamma\delta(A_{\centerdot})_{\centerdot p}}  && \Sigma_n\Sigma_m\delta(A_m)_{n,p} \\
\Sigma_mA_m \ar[rr]_-{\lambda_p\Sigma_mA_m} \ar[u]^-{\Sigma_m\lambda_pA_m} && \Sigma_n\delta(\Sigma_mA_m)_{n,p}  \ar[u]_-{\Sigma_n\overline{\lambda_00}}
}
\end{aligned}
\end{equation}

Just as Proposition~\ref{commutative} used \eqref{sumswap} and \eqref{zerocond}, we can use \eqref{smoncatdata} to obtain a canonical isomorphism
\begin{eqnarray}\label{xihat}
\widehat{\xi}\dd \Sigma_rA_{\xi(r)}\cong  \Sigma_nA_n 
\end{eqnarray}
for any injective function
 $\xi \dd \mathbb{N} \to \mathbb{N}$ and any
 $A\in \CA^\mathbb{N}$ with $A_n=0$ for $n$ not in the image of $\xi$.
 
 If $\xi \dd \mathbb{N} \to \mathbb{N}\times \mathbb{N}$ is an injective function and $A_{m,n}=0$ for $(m,n)$ not in the image of $\xi$,
 then we have a canonical isomorphism
 $$\widehat{\xi}\dd \Sigma_rA_{\xi(r)}\cong\Sigma_m\Sigma_nA_{(m,n)} \ .$$
 
 Clearly the dual $\CA^{\mathrm{op}}$ of a series monoidal category $\CA$
 is series monoidal with the same $\Sigma$ and $0$. 
 
 \begin{example}\label{countcoprodex}
 {\em Any category $\CA$ with countable coproducts is series monoidal with $\Sigma$ taken to be the coproduct.
 Dually, any category with countable products is series monoidal.
 For $\CA=\mathrm{SerMn}$, these two series monoidal structures coincide (by Proposition~\ref{countabledirectsum}).}
 \end{example}
 
  \begin{example}\label{exsuplatticecat}
 {\em Of course every partially ordered set is a category with at most one morphism between one object and another.
 This category structure is compatible with the series monoid structure of the countable-sup-lattice example (Example~\ref{exsuplattice}) and so gives a series monoidal category. This is actually a special case of Example~\ref{countcoprodex}.}
 \end{example}   
 
   \begin{example}\label{sermonascat}
 {\em Indeed, every series monoid $A$ is a series monoidal category by regarding it as a discrete category.
 Also $A$ is a series monoidal category by regarding it as a category using the pre-order defined in Section~\ref{sermagmon};
 Higgs \cite{Higgs1978} proves this is a partial order when $A$ is a magnitude module.}
 \end{example}  
 
 \begin{example}\label{inftenRmod}
 {\em {\blue This is not an example, although claimed to be so in the published version of this paper.} 
 Let $R$ be a commutative ring and consider the category $\mathrm{Mod}_R$ of $R$-modules.
 For $A\in \mathrm{Mod}_R^{\mathbb{N}}$ and any $R$-module $B$, define a function
 $f\dd \prod_{n\in \mathrm{N}}{A_n}\to B$ to be {\em multilinear} when each 
 $$f(a_0,a_1,\dots, a_{m-1},-,a_{m+1}, \dots)\dd A_m \to B$$
 is an $R$-module morphism.
 The representing countable tensor does not give a series monoidal structure. 
 This multiple tensor is studied in Chevalley's book [Fundamental Concepts of Algebra (Academic Press, 1956)]. 
 For $R=\mathbb{C}$, Ng \cite{Ng2013} makes some use of this tensor product, along with some variants.}
 \end{example}
 
  \begin{example}\label{inftenRsermod}
 {\em {\blue This is not an example, although claimed to be so in the published version of this paper.} 
 For a commutative monoid $R$ in the monoidal category $\mathrm{SerMn}$ (Section~\ref{scs}), 
 we purported to have a multilinear-style series monoidal structure on the category 
 $\mathrm{SerMd}_R$ of $R$-modules, and, in particular, on $\mathrm{SerMn}$. We do not.}
 \end{example}
 
  \begin{example}\label{funnyex}
 {\em For a sequence $A=(A_n)_{n\in \mathbb{N}}$ of small categories
 and a category $X$, a {\em funny functor} $f\dd A \to X$ is a function
 assigning to each object $a\in \prod_nA_n$ an object $f(a)\in X$,
 equipped with the structure of a functor $A_m \to X$ on each object
 assignment $a_m \mapsto f(a)$ with all $a_n\in A_n$ fixed for $n\neq m$.
 There is a category $\ddot\smile A =\ddot\smile_{n\in \mathbb{N}} A_n$
 such that funny functors $A\to X$ are in natural bijection with functors
 $\ddot\smile A \to X$. There is a series monoidal structure on the
 category $\mathrm{Cat}$ of small categories where $\Sigma = \ddot\smile$.     
 }
 \end{example}
 
 We now make the natural definition of series monoidal functor.
 \begin{definition}
 {\em Suppose $\CA$ and $\CX$ are series monoidal categories.
 A functor $F\dd \CA\to \CX$ is {\em series monoidal} when it is equipped
 with a morphism $\phi_0\dd 0\to F0$ in $\CX$ and a natural transformation
 with components
 $$\phi A \dd \Sigma_nFA_n \lra F\Sigma_nA_n$$
 such that diagrams \eqref{presgamma} and \eqref{preslambda} commute.
 We call $F$ {\em series strong monoidal} when $\phi$ and $\phi_0$
 are invertible. 

A series monoidal functor $M\dd 1\to \CA$ is called a {\em series monoid in $\CA$};
that is, $M$ is an object of $\CA$ equipped with morphisms $s_0\dd 0\to M$ and
$s\dd \mathbb{N}\cdot M = \Sigma (M,M,\dots )\to M$ subject to the two conditions \eqref{presgamma} 
and \eqref{preslambda} with $\phi_0=s_0$ and $\phi * = s$ for $*\in 1$.
Since series monoidal functors compose, they take series monoids to series monoids.
}
 \end{definition} 
 \begin{equation}\label{presgamma}
\begin{aligned}
\xymatrix{
& \Sigma_mF\Sigma_nA_{mn} \ar[rd]^-{\phantom{A}\phi\Sigma_n A_{\centerdot n}}  & \\
\Sigma_m\Sigma_nFA_{mn} \ar[ru]^-{\Sigma_m\phi A_{m\centerdot}\phantom{A}} \ar[d]_-{\gamma F A} & & F\Sigma_m\Sigma_nA_{mn} \ar[d]^-{F\gamma A} \\
\Sigma_n\Sigma_mFA_{mn} \ar[rd]_-{\Sigma_n\phi A_{\centerdot n}} & & F\Sigma_m\Sigma_nA_{mn}  \\
& \Sigma_nF\Sigma_mA_{mnp}  \ar[ru]_-{\phi\Sigma_m A_{m\centerdot}} 
}
\end{aligned}
\end{equation}
\begin{equation}\label{preslambda}
\begin{aligned}
\xymatrix{
F\Sigma_n\delta(A)_{n,p} \ar[rr]^-{\phi\delta(A)_{\centerdot p}}  && \Sigma_nF\delta(A)_{n,p}  \\
FA \ar[rr]_-{\lambda_pFA} \ar[u]^-{F\lambda_pA} && \Sigma_n\delta(FA)_{n,p} \ar[u]_-{\Sigma_n\overline{\phi_0}}}
\end{aligned}
\end{equation}

\begin{example}\label{sermnhom}
{\em For any series monoidal category $\CA$, the hom functor
$$\CA(-.-)\dd \CA^{\mathrm{op}}\times \CA \lra \mathrm{Set}$$
is series monoidal where the series monoidal structure on
$\mathrm{Set}$ is countable product. Here $\phi_0 \dd \mathbf{1}\to \CA(0,0)$
picks out the identity morphism of $0\in \CA$ while $\phi(C,A)$ is the effect
$$\prod_{n\in \mathbb{N}}\CA(C_n,A_n) = \CA^{\mathbb{N}}(C,A)\xra{\Sigma} \CA(\Sigma_{n\in \mathbb{N}} C_n,\Sigma_{n\in \mathbb{N}} A_n) $$
of the functor $\Sigma$ on homs.
It follows that, if $C$ is a {\em series comonoid in $\CA$} (= series monoid in $\CA^{\mathrm{op}}$)
and $A$ is a series monoid in $\CA$ (so that $(C,A)$ is a series monoid in $\CA^{\mathrm{op}}\times \CA$),
then $\CA(C,A)$ becomes a series monoid in $\mathrm{Set}$; naturally this is called {\em convolution}.}   
\end{example}

 \begin{definition}
 {\em Suppose $F, G\dd \CA \to \CX$ are series monoidal functors.
 A natural transformation $\sigma \dd F\Rightarrow G $ is {\em series monoidal} when the
 two diagrams \eqref{smonnat} commute.}
 \end{definition}
 \begin{equation}\label{smonnat}
\begin{aligned}
\xymatrix{
\Sigma_nFA_n \ar[rr]^-{\Sigma_n\sigma A_n} \ar[d]_-{\phi A} && \Sigma_nGA_n \ar[d]^-{\phi A} \\
F\Sigma_nA_n \ar[rr]_-{\sigma\Sigma A_n} && G\Sigma_nA_n}
\quad \quad
\xymatrix{
0 \ar[rd]_{\phi_0}\ar[rr]^{\phi_0}   && F0 \ar[ld]^{\sigma 0} \\
& G0  &
}
\end{aligned}
\end{equation} 

With the obvious compositions, this defines a 2-category $\mathrm{SerMnCat}$.
Write $\mathrm{SerMn_sCat}$ for the sub-2-category obtained by restricting to
the series strong monoidal functors.
The 2-category $\mathrm{SerMnCat}$ has products preserved by the forgetful
2-functor into $\mathrm{Cat}$. 
It is immediate from the definitions that:
\begin{proposition}\label{2EckHilt}
For any series monoidal category $\CA$, the functor 
$$\Sigma \dd \CA^{\mathbb{N}}\to \CA$$
is series strong monoidal.
\end{proposition}

Associated with this kind of ``commutativity'' of the theory is the fact that any countable product of 
series monoidal categories is also the bicategorical coproduct in $\mathrm{SerMn_sCat}$;
that is, $\mathrm{SerMn_sCat}$ has countable direct sums in the bicategorical sense.  

\begin{example}\label{Rmodadj}
{\em As this involved the non-Example~\ref{inftenRmod}, it was in error.}
\end{example}

\begin{example}\label{Rsermodadj}
{\em As this involved the non-Example~\ref{inftenRsermod}, it was also in error.}
\end{example}

For any series monoidal category $\CA$ and subset $S\subseteq \mathbb{N}$, 
we can define a functor $\Sigma_S\dd \CA^S\to \CA$ whose value at $A\in \CA^S$ is 
\begin{eqnarray}
\Sigma_{n\in S}A_n=\Sigma_{n\in \mathbb{N}}C_n
\end{eqnarray}
where
 \begin{equation*}
C_n =
\begin{cases}
A_n & \text{if }  n\in S\\
0 & \text{otherwise}.
\end{cases}
\end{equation*}
When $A_n=B$ for all $n\in S$, we also put 
\begin{eqnarray}\label{actionSonB}
S\cdot B : = \Sigma_{n\in S}A_n \ .
\end{eqnarray}
Using \eqref{xihat}, we obtain, for any bijection $\xi\dd S\to T$,
an isomorphism
\begin{eqnarray}\label{bijB}
\widehat{\xi}\dd \Sigma_{r\in S}A_{\xi(r)} \cong \Sigma_{n\in T}A_n
\end{eqnarray}
with the special case
\begin{eqnarray}\label{bijB}
\widehat{\xi}\dd S\cdot B \cong T\cdot B \ .
\end{eqnarray}
This definition transports to any bijection $\xi$ between any countable
sets $S$ and $T$ yielding a functor
\begin{eqnarray}\label{CFaction}
-\cdot - \dd \mathrm{CF}\times \CA \lra \CA , \ (S,B) \mapsto S\cdot B \ , 
\end{eqnarray}
where $\mathrm{CF}$ is the category of countable sets and all functions. 
Notice that, for $S,T\subseteq \mathbb{N}$ and $A =(A_{m,n})_{(m,n)\in S\times T}$,
the isomorphism $\gamma$ of \eqref{smoncatdata} restricts to an isomorphism
\begin{eqnarray*}
\gamma \dd \Sigma_{m\in S}\Sigma_{n\in T}A_{m,n}\cong \Sigma_{n\in T}\Sigma_{m\in S}A_{m,n} \ .
\end{eqnarray*}
When $S$ is any countable set and $A =(A_m)_{m\in T}$, this transports to an isomorphism
\begin{eqnarray}\label{SdotSigma}
\gamma \dd S\cdot \Sigma_{n\in T}A_{n}\cong \Sigma_{n\in T}S\cdot A_{n} \ .
\end{eqnarray}
For any countable $T$ and $A\in \CA$, we obtain
\begin{eqnarray*}
\gamma \dd S\cdot (T\cdot A)\cong T\cdot (S\cdot A) \ ,
\end{eqnarray*}
and this is isomorphic to $(S\times T)\cdot A$.

Write $\mathrm{ev}_S\dd S\cdot A^S\to A$ for the morphism corresponding
to the identity of $A^S$. Then each bijection $\xi\dd S\to T$ determines an
isomorphism 
\begin{eqnarray}\label{xibar}
\overline{\xi}\dd A^S\to A^T
\end{eqnarray}
which corresponds to the composite $T\cdot A^S\xra{\widehat{\xi}^{-1}\cdot 1}S\cdot A^S \xra{\mathrm{ev}_S}A$.  

\begin{proposition}\label{symmonadd}
The tensor product defined by 
\begin{eqnarray}
A_1+A_2 = \Sigma_{n\in \{1,2\}}A_n \ .
\end{eqnarray}
renders $\CA$ symmetric monoidal with $0$ as tensor unit.
Moreover, $\Sigma : \CA^{\mathbb{N}}\to~\CA$ is a symmetric strong monoidal functor.
\end{proposition}

  For series monoidal categories $\CA$ and $\CX$, the category $\mathrm{SerMn_sCat}(\CA,\CX)$
  is series monoidal under the pointwise series monoidal structure; we write $\mathrm{Ser}(\CA,\CX)$ for
  this series monoidal category. 
  For a sequence $F=(F_n)_{n\in \mathbb{N}}$ of series strong monoidal functors $F_n\dd \CA\to \CX$,
  the definition of $\Sigma F$ is the composite
  \begin{eqnarray}\label{smcinthom}
\CA \xra{(F_n)}\CX^\mathbb{N}\xra \Sigma \CX \ .
\end{eqnarray}

The forgetful 2-functor 
\begin{eqnarray}\label{SerMnCatU}
\mathrm{U} \dd \mathrm{SerMn_sCat}\to \mathrm{Cat}
\end{eqnarray}
  is monadic in a bicategorical sense. In particular, it has a left biadjoint (see \cite{BKP}
  for this sort of result) whose value
  at the terminal category can be made explicit.
 
 \begin{proposition}\label{freesmoncat1}
 The 2-functor \eqref{SerMnCatU} is pseudo-representable by the series monoidal
 category $\mathrm{CB}$ with countable sets as objects, bijective functions as morphisms, 
 and disjoint union as $\Sigma$. 
 To be precise, for any series monoidal category $\CA$, the category
 of series strong monoidal functors $\mathrm{CB}\to \CA$
 is pseudonaturally equivalent to the category $\CA$.
 A series monoidal equivalence 
 $$\mathrm{Ser}(\mathrm{CB},\CA) \simeq \CA \ ,$$
defined by evaluation at the singleton set, follows therefrom.
  \end{proposition}
   \begin{proof}
   Given an object $A\in \CA$, we define a series strong monoidal functor
   $F\dd \mathrm{CB}\to\CA$ with $F\mathbf{1}=A$ as follows. Define 
   $F-=-\cdot A$ as per \eqref{CFaction} with series monoidal structure
supplied by \eqref{SdotSigma}. 
The assignment $A\mapsto F$ is the object function for a functor 
$\CA\to \mathrm{Ser}(\mathrm{CB},\CA)$ defined on morphisms
by universality. 
This provides the inverse equivalence to evaluation at $\mathbf{1}$.
    \end{proof}
  We also have the 2-functor
  \begin{eqnarray}\label{sermn}
\mathrm{sermn}\dd \mathrm{SerMn_sCat} \lra \mathrm{Cat}
\end{eqnarray}
which takes each series monoidal category $\CA$ to the category 
\begin{eqnarray*}
\mathrm{sermn}\CA = \mathrm{SerMnCat}(\mathbf{1},\CA)
\end{eqnarray*}
of series monoids in $\CA$.

  \begin{proposition}\label{freesmoncat2}
 The 2-functor \eqref{sermn} is pseudo-representable by the series monoidal
 category $\mathrm{CF}$ with countable sets as objects, functions as morphisms, 
 and disjoint union (coproduct) as $\Sigma$. 
 To be precise, an equivalence of categories 
 $$\mathrm{SerMn_sCat}(\mathrm{CF},\CA) \simeq \mathrm{sermn}\CA \ ,$$
 pseudonatural in series monoidal categories $\CA$.
  \end{proposition}
   \begin{proof}
   Given a series monoid $A\in \CA$, we have the series strong monoidal functor
   $F\dd \mathrm{CB}\to\CA$ with $F\mathbf{1}=A$ as in Proposition~\ref{freesmoncat1}. 
   Using the series monoid structure $s_0\dd 1\to A$, $s\dd \mathbb{N}\cdot A\to A$ on $A$,
   we can extend $F$ to a series strong monoidal functor $F'\dd \mathrm{CF}\to\CA$ as
   follows. For any $S\subseteq \mathbb{N}$, let $A_n=A$ for all $n\in \mathbb{N}$, 
   let $C_n=A$ for all $n\in S$, let $C_n=1$ for all $n\notin S$,
   and let $u_n\dd C_n \to A_n$ be the identity of $A$ for $n\in S$ and $s_0$ otherwise.     
   We can define 
   \begin{eqnarray*}
\left(S\cdot A\xra{s_S} A\right) = \left( \Sigma_nC_n\xra{\Sigma_nu_n}\Sigma_nA_n\xra{s}A\right) \ .
\end{eqnarray*}
For any order-preserving function $\alpha \dd S\to T$ between subsets of $\mathbb{N}$, we obtain
$$(S\cdot A \xra{s_{\alpha}}T\cdot A) : = (\Sigma_{n\in T}\alpha^{-1}(n)\cdot A\xra{\Sigma_{n\in T}s_{\alpha^{-1}(n)}}\Sigma_{n\in T}A) \ .$$ 
For a bijective $\xi\dd S\to S$, we already have $\widehat{\xi}\dd S\cdot A\cong S\cdot A$ as in \eqref{bijB}.
As every function $\alpha\dd S\to T$ is a composite of an automorphism $\alpha_1 \dd S\to S$ and an order-preserving 
function $\alpha_2 \dd S\to T$, we obtain
$$F\alpha = s_{\alpha} = s_{\alpha_2}\circ \widehat{\alpha_1}\dd S\cdot A \to T\cdot A$$ 
in $\CA$.
The remaining details of the proof that this gives an inverse equivalence are as for finite sets, 
symmetric monoidal categories, and commutative monoids.
\end{proof}
    
 As a bicategory, $\mathrm{SerMn_sCat}$ is symmetric closed monoidal in the sense of
 \cite{60}. There is a tensor product $\CA\otimes \CB$ satisfying pseudonatural equivalences
 \begin{align}\label{Ser symmon}
\mathrm{SerMn_sCat}(\CA,\mathrm{Ser}(\CB , \CX)) &  \simeq \mathrm{SerMn_sCat}(\CA \otimes \CB, \CX) \nonumber \\ 
&  \simeq \mathrm{SerMn_sCat}(\CB,\mathrm{Ser}(\CA,\CX)) \ . 
\end{align}

A {\em diagonal} for an object $A$ of a series monoidal category $\CA$
is a morphism $d\dd A\lra \Sigma(A,A,\dots )$ such that, for all bijections 
$\xi\dd \mathbb{N}\to \mathbb{N}\times \mathbb{N}$, the following square commutes.
\begin{equation}
\begin{aligned}
\xymatrix{
A \ar[rr]^-{d} \ar[d]_-{d} && \Sigma(A,A,\dots ) \ar[d]^-{\widehat{\xi}} \\
\Sigma(A,A,\dots ) \ar[rr]_-{\Sigma (d,d,\dots )} && \Sigma (\Sigma(A,A,\dots) ,\Sigma(A,A,\dots ),\dots )}
\end{aligned}
\end{equation}

\begin{definition}
{\em A {\em magnitude module} $M$ in a series monoidal category $\CA$ is a series monoid equipped
with a diagonal morphism $d\dd M\lra \Sigma(M,M,\dots )$ and a series monoid endomorphism $h\dd M\to M$
such that the composite
\begin{eqnarray}
\tilde{h}\dd M\xra{d}\Sigma(M,M,\dots )\xra{(h,h\circ h,h\circ h\circ h, \dots)}\Sigma(M,M,\dots )\xra{s}M
\end{eqnarray}
is the identity of $M$. }
\end{definition}
   
 \section{Zeno functors and magnitude categories}\label{Zemc}
 
 Let $F\dd \CA \to \CA$ be a series monoidal endofunctor on the series monoidal
 category $\CA$.
 For each $n\in \mathbb{N}$, we have the $n$-fold composite series monoidal endofunctor
 $$F^{\circ n} = \overbrace{F\circ F \circ \dots F}^n \dd \CA\lra \CA \ .$$ 
By the product property of $\CA^{\mathbb{N}}$ in $\mathrm{SerMnCat}$,
a series monoidal functor 
$$F^{\circ (\centerdot + 1)} \dd \CA\lra \CA^{\mathbb{N}}$$
is induced. 
This composes with the series strong monoidal functor $\Sigma$ of
Proposition~\ref{2EckHilt} to yield a series monoidal functor
\begin{eqnarray}
\tilde{F}=\Sigma_{n\in \mathrm{N}}F^{\circ (n + 1)} \dd \CA\lra \CA \ .
\end{eqnarray}
There are canonical natural isomorphisms
\begin{eqnarray}\label{canonH}
F + F\circ \tilde{F} \cong \tilde{F} \ \text{ and } \  \tilde{F} \circ F\cong F\circ \tilde{F} \ .
\end{eqnarray}
 \begin{definition} {\em A {\em Zeno functor} on $\CA$ is a series strong monoidal functor $H\dd \CA\to \CA$ 
 equipped with a series monoidal isomorphism $\kappa\dd 1_{\CA}\cong \tilde{H}$ such that \eqref{wellpt} commutes.
 A {\em magnitude category} is a
 series monoidal category equipped with a Zeno functor $H$.} 
\begin{eqnarray}\label{wellpt}
\begin{aligned}
\xymatrix{
H \ar[rd]_{\kappa \circ H}\ar[rr]^{H\circ \kappa }   && H\circ \tilde{H}  \\
& \tilde{H} \circ H \ar[ru]_{\text{\eqref{canonH}}} &
}
\end{aligned}
\end{eqnarray}
\end{definition}

\begin{example}\label{exnathalfcat}
{\em When $\CA=\mathrm{CB}$ or $\CA=\mathrm{CF}$, there exists no Zeno functor since 
$1$ is not isomorphic to a disjoint union of a set with itself.}
\end{example}

\begin{example}\label{exrealhalfcat}
{\em Any magnitude module, either as a discrete category or with its partial order,
is a magnitude category.}
\end{example}

\begin{example}\label{functormagcats}
{\em For any category $\CC$ and magnitude category $\CA$, the functor category
$[\CC,\CA]$ is a magnitude category with the pointwise structure.
If $\CC$ is serial monoidal then $\mathrm{Ser}(\CC,\CA)$ is a magnitude subcategory
of $[\CC,\CA]$.}
\end{example}

 \begin{definition} {\em A {\em magnitude functor} $F\dd \CA \to \CX$ is a
 series monoidal functor equipped with a 
 series monoidal natural transformation $\nu_1\dd H\circ F\Rightarrow F\circ H$
 compatible with the series monoidal isomorphisms $\kappa\dd 1_{\CA}\cong \tilde{H}$
 in the sense that \eqref{magfuncond} should commute.
 It is {\em strong} when it is series strong monoidal and $\nu_1$ is invertible.} 
\end{definition}

The natural transformation $\nu_1\dd H\circ F\Rightarrow F\circ H$ inductively determines
natural transformations $\nu_n\dd H^{\circ n}\circ F\Rightarrow F\circ H^{\circ n}$ via
\begin{eqnarray*}
\nu_{n+1}\dd H^{\circ (n + 1)}\circ F\xra{H^{\circ n} \nu_1}H^{\circ n}\circ F\circ H\xra{\nu_n\circ H}F\circ H^{\circ (n + 1)} \ ,
\end{eqnarray*}
and hence a natural transformation 
$$\tilde{\nu} = \Sigma_{n>0}\nu_n \dd F\circ \tilde{H}\Rightarrow \tilde{H}\circ F \ .$$ 
We ask commutativity of
\begin{eqnarray}\label{magfuncond}
\begin{aligned}
\xymatrix{
F \ar[rd]_{\kappa \circ F}\ar[rr]^{F\circ \kappa }   && F\circ \tilde{H}  \\
& \tilde{H} \circ F \ar[ru]_{\tilde{\nu}} & \ .
}
\end{aligned}
\end{eqnarray}
\begin{example}
{\em For a magnitude category $\CA$, the Zeno functor $H\dd \CA\to \CA$ is a magnitude functor
with $\nu_1=1_{H\circ H}$.}
\end{example}

\begin{definition}
 {\em Suppose $F, G\dd \CA \to \CX$ are magnitude functors.
 A {\em magnitude natural transformation} $\sigma \dd F\Rightarrow G$ is a series monoidal natural
 transformation for which the following square commutes.
 
 \begin{equation*}
\xymatrix{
H\circ F \ar[rr]^-{\nu_1} \ar[d]_-{H\sigma} && F\circ H \ar[d]^-{\sigma H} \\
H\circ G \ar[rr]_-{\nu_1} && G\circ H}
\end{equation*}} 
\end{definition}

With the obvious compositions, this defines a 2-category $\mathrm{MgnCat}$ of
magnitude categories, magnitude functors and magnitude natural transformations. 
We write $\mathrm{Mgn_sCat}$ for the sub-2-category obtained by restricting
to strong magnitude functors.

If $\CC$ and $\CA$ are magnitude categories then the category 
$\mathrm{Mgn}(\CC,\CA)$ 
of strong magnitude functors and magnitude natural transformations is
a magnitude subcategory of $\mathrm{Ser}(\CC,\CA)$.

The countable direct sums of $\mathrm{SerMn_sCat}$ restrict to $\mathrm{Mgn_sCat}$;
that is, countable products restrict and the appropriate coproduct injections are magnitude functors.  

The forgetful 2-functor $$\mathrm{U} \dd \mathrm{Mgn_sCat}\to \mathrm{Cat}$$
is monadic in the bicategorical sense. In particular, it has a left biadjoint.
  
 \begin{definition}
 {\em The value at the terminal category of the left biadjoint to $\mathrm{U}$
 is called {\em the magnitude groupoid of positive real sets} and denoted
 by $\mathrm{RSet}_{\mathrm{g}}$. } 
\end{definition}

Since $\mathrm{RSet}_{\mathrm{g}}$ is series monoidal, by Proposition~\ref{freesmoncat1}
there is a series strong monoidal functor
$$\mathrm{I}\dd \mathrm{CB}\lra \mathrm{RSet}_{\mathrm{g}} $$ 
from the series monoidal category of countable sets for which 
$\mathrm{I}(\mathbf{1})$ is the generator of $\mathrm{RSet}_{\mathrm{g}}$,
which generator we shall also denote by $\mathbf{1}$.
Indeed, we shall put $\mathbf{n}\dd =\mathrm{I}(\mathbf{n})$.
We conjecture that $\mathrm{I}$ is faithful. 

There is also a strong magnitude functor
 $$\mathrm{\#}\dd \mathrm{RSet}_{\mathrm{g}} \lra [0,\infty] \ ,$$
 called {\em cardinality}, taking the generator $\mathbf{1}$ to the real number $1$.
 It follows that the composite
 $$\mathrm{CB}\xra{\mathrm{I}} \mathrm{RSet}_{\mathrm{g}} \xra{\#} [0,\infty]$$
 takes each countable set to its cardinality.
 
 For any magnitude category $\CA$, by freeness there is a strong magnitude functor
$$L\dd \mathrm{RSet}_{\mathrm{g}} \lra \mathrm{Mgn}(\CA,\CA)$$  
taking the generator $\mathbf{1} \in \mathrm{RSet}_{\mathrm{g}}$ to the identity functor of $\CA$;
see \eqref{smcinthom}.
This gives an action 
$$\bullet \ \dd \mathrm{RSet}_{\mathrm{g}}\otimes \CA \lra \CA$$
of $\mathrm{RSet}_{\mathrm{g}}$ on $\CA$ defined by $S\bullet A = (LS)A$.  
In particular, we have a monoidal structure
\begin{eqnarray}\label{multRSet}
\bullet \ \dd \mathrm{RSet}_{\mathrm{g}}\otimes \mathrm{RSet}_{\mathrm{g}} \lra \mathrm{RSet}_{\mathrm{g}}
\end{eqnarray}
on $\mathrm{RSet}_{\mathrm{g}}$ in the monoidal bicategory $\mathrm{SerMn_sCat}$; 
the unit is the generator $\mathbf{1}$ of $\mathrm{RSet}_{\mathrm{g}}$.

\begin{proposition} 
The pseudoalgebras for the pseudomonad $\mathrm{RSet}_{\mathrm{g}}\otimes -$
on $\mathrm{SerMn_sCat}$ are the magnitude categories. 
\end{proposition}
 
 We define some objects of $\mathrm{RSet}_{\mathrm{g}}$ by
 $$\mathbf{\frac{1}{2}} =H(\mathbf{1}) \ \text{,   } \frac{\mathbf{1}}{\mathbf{2}^n} =H^{\circ n}(\mathbf{1})  \ \text{,   } \mathbf{\frac{1}{3}} =\Sigma_{n>0}\frac{\mathbf{1}}{\mathbf{2}^{2n}} \ \text{,   }  \mathbf{\frac{1}{4}} =\frac{\mathbf{1}}{\mathbf{2}^{2}} \ \text{,   }  \frac{\mathbf{m}}{\mathbf{2}^{2n}} = \mathbf{m}\bullet  \frac{\mathbf{1}}{\mathbf{2}^{2n}} \ , $$ 
and so on.
For any natural number $k$, let $t$ be the first natural number with $k\le 2^t$ and can define
$$\frac{\mathbf{1}}{\mathbf{k}} = \Sigma_{n\in \mathbb{N}}\left(\frac{\mathbf{(2^t-k)}^n}{\mathbf{2}^{t(n+1)}}\right) \ . $$ 

More typically, to obtain an object of cardinality $\pi$,
express $\pi - 3 = 0.a_1a_2\dots$ in binary form, 
 let $a_{m_n}=1$ be the $n$th non-zero term in that expansion. Then 
 $$\Pi= \mathbf{3} +  \frac{\mathbf{1}}{\mathbf{2}^{m_1}} + \frac{\mathbf{1}}{\mathbf{2}^{m_2}} + \dots$$
 is an object of $\mathrm{RSet}_{\mathrm{g}}$ with $\#\Pi = \pi$.   
 A more difficult question is whether $\Pi$ has interesting automorphisms.   

For any magnitude category $\CA$, we can define an {\em exponential functor}
 \begin{eqnarray}
E\dd \CA \lra \CA
\end{eqnarray}
by 
\begin{eqnarray}
E(X)= \Sigma_{n\in \mathbb{N}}\frac{\mathbf{1}}{\mathbf{n!}}\bullet \ X^{\bullet n} \ .
\end{eqnarray}

We also have the 2-functor
  \begin{eqnarray}\label{sermnmgn}
\mathrm{sermn}\dd \mathrm{Mgn_sCat} \lra \mathrm{Cat}
\end{eqnarray}
which takes each magnitude category $\CA$ to the category 
$\mathrm{sermn}\CA$ of series monoids in $\CA$.

  \begin{definition}\label{RSet}
   {\em The pseudo-representing object for the 2-functor \eqref{sermnmgn} is called 
   {\em the magnitude category of positive real sets} and denoted by $\mathrm{RSet}$.
 That is, there is an equivalence of categories 
 $$\mathrm{Mgn_sCat}(\mathrm{RSet},\CA) \cong \mathrm{sermn}\CA \ ,$$
 pseudonatural in magnitude categories $\CA$.}
  \end{definition}
  
  There is a magnitude functor $\mathrm{RSet}_g \lra \mathrm{RSet}$ taking $\mathbf{1}$
  to the generator $\mathbf{1}$ of $\mathrm{RSet}$.

\begin{remark}\label{catexplicitconstruct}  
{\em Here is a construction of $\mathrm{RSet}$ in the spirit of Remark~\ref{explicitconstruct}.
Begin with the pointwise series monoidal category $\CR_1=\mathrm{CF}^{\mathbb{N}}$
of sequences of countable sets. We have a series strict monoidal functor $H_1\dd \CR_1\to \CR_1$
defined as the suspension $H_1(X_0,X_1,X_2,\dots)=(0, X_0,X_1,\dots)$. 
We then form the series strict monoidal functor $\tilde{H_1}$; the formula is 
$$\tilde{H_1}(X_0,X_1,\dots)=(0,X_0,X_0+X_1,\dots)=(\Sigma_{n<0}X_n,\Sigma_{n<1}X_n,\Sigma_{n<2}X_n,\dots) \ .$$
Note that $\tilde{H_1}\circ H_1=H_1\circ \tilde{H_1}$.
Now form the isocoinserter \eqref{isocoins} of the identity functor of $\CR_1$ and $\tilde{H_1}$ in the 2-category $\mathrm{SerMn_sCat}$.
\begin{equation}\label{isocoins}
\begin{aligned}
\xymatrix{
\CR_1 \ar[d]_{1_{\CR_1}}^(0.5){\phantom{aaaaa}}="1" \ar[rr]^{\tilde{H_1}}  && \CR_1 \ar[d]^{P_1}_(0.5){\phantom{aaaaa}}="2" \ar@{=>}"1";"2"^-{\kappa_1}
\\
\CR_1 \ar[rr]_-{P_1} && \CR_2 
}
\end{aligned}
\end{equation}
Since we have the isomorphism $P_1\circ H_1\xra{\kappa_1H_1}P_1\circ \tilde{H_1}\circ H_1=P_1\circ H_1\circ \tilde{H_1}$,
the universal property of \eqref{isocoins} yields a unique series strong monoidal functor $H_2\dd \CR_2\to\CR_2$
such that $H_2\circ P_1=P_1\circ H_1$ and $H_2\circ \kappa_1 = \kappa_1\circ H_1$. 
It follows that $\tilde{H_2}\circ P_1=P_1\circ \tilde{H_1}$ and $\tilde{H_2}\circ \kappa_1 = \kappa_1\circ \tilde{H_1}$.  
So the universal property of \eqref{isocoins} yields a unique series strong monoidal natural isomorphism
$\kappa_2\dd 1_{\CR_2}\Rightarrow \tilde{H_2}$ such that $\kappa_2\circ P_1= \kappa_1$.
Note also that $\tilde{H_2}\circ H_2=H_2\circ \tilde{H_2}$.
So we have two 2-cells $H_2\circ \kappa_2$ and $\kappa_2\circ H_2$ from $H_2$ to $\tilde{H_2}\circ H_2$
and we can take their coequifier $P_2\dd \CR_2\to \CR$ in the 2-category $\mathrm{SerMn_sCat}$.
From the universal property of the coequifier, we obtain a unique series strong monoidal functor $H\dd \CR \to \CR$
with $P_2\circ H_2=H\circ P_2$, and then obtain a unique series monoidal natural isomorphism 
$\kappa \dd 1_{\CR}\Rightarrow \tilde{H}$ with $\kappa \circ P_2=P_2\circ\kappa_2$.
Now \eqref{wellpt} is satisfied and we have a Zeno functor $(H,\kappa)$ making $\CR$ a magnitude category.}
\end{remark}

\begin{proposition}\label{propositiononexplconst}
The magnitude category $\CR$ constructed in Remark~\ref{catexplicitconstruct} is equivalent to $\mathrm{RSet}$. 
\end{proposition}
\begin{proof}
The universal properties of $\CR_2$ and $\CR$ combine to show that strong magnitude functors
$F\dd \CR\to \CA$ are in bijection with series strong monoidal functors $G\dd \CR_1\to \CA$ equipped with
a series monoidal isomorphism $$\nu_1\dd H\circ G\Ra G\circ H_1 \ .$$
However, $\CR_1=\mathrm{CF}^{\mathrm{N}}$ is the coproduct of countably many copies of $\mathrm{CF}$.
So, to give $G$ is equivalently to give a sequence of series strong monoidal functors $G_n\dd \mathrm{CF}\to \CA$.
By Proposition~\ref{freesmoncat2}, to give such a sequence is equivalent to giving a sequence of series monoids $A_n$ in $\CA$. 
However, $\nu_1$ induces a series monoid isomorphism $\nu_n\dd HA_n\cong A_{n+1}$.
So the sequence of series monoids is, up to canonical isomorphism, determined by the single series monoid $A=A_0$. 
\end{proof}

\section{Remarks on integer sets}\label{rois}

To obtain reals from integers, Higgs taught us to introduce a halving operation.
It is obvious to all that to obtain integers from natural numbers, 
we need to introduce a minus operation.
A categorical version of minus might be dual. If we think of the categorical integers
as forming the free symmetric monoidal category on a single generating object,
we might think of the categorical integers as forming a compact closed
category in the sense of Kelly \cite{KellyMvfcI}; this includes symmetry.

Let $\mathrm{symMon}$ denote the groupoid-enriched category of symmetric
monoidal categories, symmetric strong monoidal functors, and monoidal
natural isomorphisms. Let $\mathrm{CmpClsd}$ denote the full sub-groupoid-enriched 
category of $\mathrm{symMon}$ consisting of the compact closed categories.
The inclusion $$\mathrm{CmpClsd}\lra \mathrm{symMon}$$ has a left biadjoint.
The value of this biadjoint at the category $\mathrm{FB}$ of finite sets and bijections
might be a candidate for a category $\mathrm{ZSet}$ of integer sets.

A category $\mathrm{IntRel}$ of integer sets and relations was introduced in \cite{51}.
It is the free tortile monoidal category on the symmetric traced monoidal category
$\mathrm{Rel}$ of sets and relations. The explicit description can be found in Section 6
of \cite{51}: the objects are pairs $(X,U)$ of sets and the morphisms $R\dd (X,U)\to (Y,V)$
are relations from $X+V$ to $Y+U$, while the composition uses the trace.
Trace categorizes the cancellation property of addition of natural numbers.  

\begin{lemma}\label{injectivetrace}
If a morphism $R\dd X+U \to Y+U$ in $\mathrm{Rel}$ is the graph of an injective
function $f\dd X+U \to Y+U$ and $U$ is a finite set then the trace 
$\mathrm{Tr}^U_{X,Y}(R)\dd X\to Y$ is the graph of an injective function
$\mathrm{Tr}^U(f)\dd X\to Y$. Indeed, $\mathrm{Tr}^U(f)(x) = f^{\circ n}(x)\in Y$
where $n>0$ is such that $f^{\circ m}(x)\in U$ for all $m<n$. 
If $f\dd X+U \to Y+U$ is bijective then so is $\mathrm{Tr}^U(f)\dd X\to Y$.      
\end{lemma} 
\begin{proof}
This is an easy exercise for a reader who recalls the matrix formula for $\mathrm{Tr}^U_{X,Y}(R)$
in \cite{51}. The reason there is such an $n$ is that $U$ is finite. 
\end{proof}

Therefore, we may wish to replace $\mathrm{symMon}$ by the groupoid-enriched category 
$\mathrm{symMon}_*$ of symmetric monoidal categories for which the tensor unit is initial.
Then the appropriate replacement for $\mathrm{FB}$ is the category $\mathrm{FI}$ of finite
sets and injective functions; by Lemma~\ref{injectivetrace}, both of these are traced
monoidal subcategories of $\mathrm{Rel}$. 

This produces two candidates $\mathrm{IntFI}$
and $\mathrm{IntFB}$ for categories of integer sets. The objects are pairs $(X,U)$ of finite 
sets and the morphisms $f\dd (X,U)\to (Y,V)$ are injective or bijective functions
$f\dd X+V\to Y+U$, respectively. The composite $g\circ f$ of $f\dd (X,U)\to (Y,V)$ and $g\dd (Y,V)\to (Z,W)$
is the trace $\mathrm{Tr}^V(g\# f)$ of the function $g\# f \dd X+W+V\to Z+U+V$ defined as 
follows. For $p\in X+V$,
 \begin{equation*}
g\# f(p) =
\begin{cases}
g(f(p))\in Z+V & \text{if } f(p) \in Y, \\
f(p)\in U & \text{otherwise } ;
\end{cases}
\end{equation*}  
while $g\# f(p) =g(p)\in Z+V$ for $p\in W$. 

Of course, $\mathbb{N}$ is traced monoidal as an ordered set under addition.
We have the inclusion functor $\mathrm{FB} \to \mathrm{FI}$ and the cardinality functor 
$\mathrm{FI}\to \mathbb{N}$ which are both traced symmetric strong monoidal. 
Since $\mathrm{Int}$ is a groupoid-enriched functor, we obtain symmetric strong monoidal functors
$$\mathrm{IntFB} \lra \mathrm{IntFI} \lra \mathbb{Z}$$
providing cardinalities for ``integer sets''.

A geometric approach, based on the idea that Euler characteristic extends cardinality, is that of Schanuel \cite{Schanuel1991}.  

\begin{remark}\label{euler} {\em There is a classical obstruction to having both associative infinite sums and negatives.
The only element $c$ with an inverse $-c$ for the binary addition
in a series monoid is $0$. The proof goes back to Euler:
\begin{eqnarray*}
0 = 0 + 0 + \dots = (c-c)+(c-c)+\dots = c+(-c+c)+(-c+c)+\dots = c .
\end{eqnarray*}}
\end{remark}

A different tack, suggested by Remark~\ref{DayRemark} following Day, is to note that $\mathrm{RSet}_{\mathrm{g}}$
itself with the multiplication monoidal structure \eqref{multRSet}, might be considered
up to equivalence to be not just positive but all extended real sets, with addition as the monoidal structure.
We do not know whether this structure is $*$-autonomous.

 \section{Categories enriched in a series \\ monoidal category}\label{ceismc}

Let $\CV$ denote a series monoidal category.
It becomes symmetric monoidal under binary summation according to Proposition~\ref{symmonadd}. 
The usual notion of $\CV$-category $\CA$ makes sense as per \cite{KellyBook}. 
Because of the symmetry, the opposite $\CA^{\mathrm{op}}$ and tensor product, here
written as a sum $\CA +\CB$, of enriched categories are already defined. 

What we wish to point out now is the possibility to sum series of $\CV$-categories.
By Proposition~\ref{symmonadd}, we have the symmetric strong monoidal functor 
$\Sigma \dd \CV^{\mathbb{N}}\to \CV$
and so, the symmetric strong monoidal 2-functor
 $$\Sigma_* \dd \CV^{\mathbb{N}}\text{-}\mathrm{Cat}\lra \CV\text{-}\mathrm{Cat}$$
in the notation of Eilenberg-Kelly \cite{EilKel1966}.
There is also the symmetric strong monoidal 2-functor
$$\mathrm{Q}\dd (\CV\text{-}\mathrm{Cat})^{\mathbb{N}}\lra \CV^{\mathbb{N}}\text{-}\mathrm{Cat}$$
taking a sequence $\CA=(\CA_n)_{n\in \mathbb{N}}$ of $\CV$-categories to the $\CV^{\mathbb{N}}$-category
$\mathrm{Q}\CA$ whose objects are objects of the cartesian product $\prod_{n\in \mathbb{N}}\CA_n$
and whose homs are defined by $\mathrm{Q}\CA(A,B)_n= \CA_n(A_n,B_n)$. 
Now define $\Sigma$ by composing thus:
\begin{eqnarray}\label{enrichedSigma}
\Sigma\dd (\CV\text{-}\mathrm{Cat})^{\mathbb{N}}\xra{\mathrm{Q}} \CV^{\mathbb{N}}\text{-}\mathrm{Cat}\xra{\Sigma_*} \CV\text{-}\mathrm{Cat} \ .
\end{eqnarray}
 
 Explicitly, for a sequence $\CA=(\CA_n)_{n\in \mathbb{N}}$ of $\CV$-categories, the objects of 
 $\Sigma \CA$ are families $A=(A_n)_{n\in \mathbb{N}}$ of objects $A_n\in \CA_n$,
 whereas the homs are defined by $\Sigma \CA(A,B)= \Sigma_n\CA_n(A_n,B_n)$.
 In the obvious sense:
 
 \begin{proposition}
 The 2-category $\CV\text{-}\mathrm{Cat}$ is series monoidal with this choice of $\Sigma$. 
 \end{proposition} 
 
  \begin{proposition}
 There is a series monoidal 2-functor $$\mathrm{SerMnCat}\to \mathrm{sMon2\text{-}Cat}$$ 
 taking $\CV$ to $\CV\text{-}\mathrm{Cat}$ and $F\dd \CV\to \CW$ to $F_*\dd \CV\text{-}\mathrm{Cat}\to \CW\text{-}\mathrm{Cat}$. 
 \end{proposition} 
 
 When our base series monoidal category $\CV$ is cocomplete in a manner allowing discussion
 of the bicategory $\CV\text{-}\mathrm{Mod}$ of $\CV$-categories and $\CV$-modules,
 the operation \eqref{enrichedSigma} extends to  
 \begin{eqnarray}\label{enrichedSigmaMod}
\Sigma\dd (\CV\text{-}\mathrm{Mod})^{\mathbb{N}}\lra \CV\text{-}\mathrm{Mod} \ .
\end{eqnarray}
This provides an example $\CV\text{-}\mathrm{Mod}$ of a series monoidal bicategory,
leading on to series promonoidal categories and convolution.

\section{$\omega$-Magmas and $\omega$-monoids}\label{omegacase}

It is natural to define monoidal categories before symmetric monoidal categories,
yet here, with the countable version, we have presented the commutative case without
mentioning the non-commutative possibility. 
The next two sections correct that omission for posterity.
 
Now we wish to pass to multiplicative terminology rather than additive. 
We call a series magma $A$ an {\em $\omega$-magma} when the operation 
$\Sigma\dd A^{\mathbb{N}}\to A$ is denoted by $\otb \dd A^{\omega}\to A$, where $\omega = \mathbb{N}$
as a linearly ordered set, and $0\in A$ is denoted by $1\in A$. 
The informal notation $\otb_{n\in \omega}a_n = a_0\ot a_1\ot \dots$ is
also helpful. 

We recall \eqref{sum/S} in the new notation. 
For any $\omega$-magma $A$ and subset $S\subseteq \omega$, 
we can define an operation $\otb_S\dd A^S\to A$ whose value at $a\in A^S$ is 
\begin{eqnarray}\label{tensor/S}
\otb_{n\in S}a_n= \otb_{n\in \omega}c_n
\end{eqnarray}
where
 \begin{equation*}
c_n =
\begin{cases}
a_n & \text{if }  n\in S\\
1 & \text{otherwise} \ .
\end{cases}
\end{equation*}

\begin{definition} {\em An {\em $\omega$-monoid} $A$ is an $\omega$-magma
such that, for all order-preserving functions $\xi\dd \omega \to \omega$, 
\begin{eqnarray}\label{genassoc}
\otb_{n\in \omega}\otb_{m\in \xi^{-1}(n)} a_m = \otb_{n\in \omega} a_n \ .
\end{eqnarray}
 } 
\end{definition} 

\begin{remark} {\em Each fibre $\xi^{-1}(n)$ of an order-preserving function $\xi\dd \omega \to \omega$
is either finite or forms a final segment of $\omega$. The latter case occurs only if $\xi$ has finite image,
and then only for the last fibre. 
} 
\end{remark}

\begin{remark} {\em After submitting the present paper, it came to our notice that \eqref{genassoc} is also the ``general associativity postulate'' (II') of Tarski \cite{Tarski1956}. He claimed it too restrictive for his purposes. 
} 
\end{remark}

\begin{example} {\em Every series monoid is an $\omega$-monoid with $\otb=\Sigma$ and $1=0$.} 
\end{example} 
 
\begin{example}\label{multP} {\em If $A$ admits an associative binary multiplication in $\mathrm{SerMn}$ 
then $A$ becomes an $\omega$-monoid with $\otb=\mathrm{P}$ (see \eqref{fP}) and $1=0$.
This was foreshadowed in Remark~\ref{foreshadomega}.} 
\end{example}

\section{$\omega$-Monoidal categories}\label{omegamoncats}

For any category $\CA$, object $I\in \CA$, and functor
\begin{eqnarray*}
\otb : \CA^{\omega} \lra \CA \ , \ (A_i)_{i\in \omega} \longmapsto \otb_{i\in \omega} A_i \ ,
\end{eqnarray*}
each subset $S\subseteq \omega$, 
determines an operation $\otb_S\dd \CA^S\to \CA$ whose value at $A\in \CA^S$ is 
\begin{eqnarray}\label{tensor/S}
\otb_{n\in S}A_n= \otb_{n\in \omega}C_n
\end{eqnarray}
where
 \begin{equation*}
C_n =
\begin{cases}
A_n & \text{if }  n\in S\\
I & \text{otherwise} \ .
\end{cases}
\end{equation*}
 
For $S, T\subseteq \omega$ and any function $\xi\dd S\to T$, define $\xi_*\dd \CA^{S}\to \CA^{T}$ by
$$\xi_*(A)_n=\otb_{\xi(m) = n} A_m$$ 
for all $A\in \CA^S$ and $n\in T$. 

\begin{definition} {\em An {\em $\omega$-monoidal category} is a category 
$\CA$ equipped with an object $I\in \CA$, a functor 
\begin{eqnarray*}
\otb : \CA^{\omega} \lra \CA \ , \ (A_i)_{i\in \omega} \longmapsto \otb_{i\in \omega} A_i
\end{eqnarray*}
and, for all order-preserving functions $\xi\dd \omega \to \omega$, natural isomorphisms

\begin{equation}\label{omegamoncatdata}
\begin{aligned}
\xymatrix{
\CA^{\omega} \ar[rd]_{\otb}^(0.5){\phantom{a}}="1" \ar[rr]^{\xi_*}  && \CA^{\omega}  \ar[ld]^{\otb}_(0.5){\phantom{a}}="2" \ar@{=>}"1";"2"^-{\alpha_{\xi}}
\\
& \CA
}
\quad
\xymatrix{
\CA \ar[rd]_{1_{\CA}}^(0.5){\phantom{a}}="1" \ar[rr]^{\delta_n}  && \CA^{\omega} \ar[ld]^{\otb}_(0.5){\phantom{a}}="2" \ar@{=>}"1";"2"^-{\lambda_n}
\\
& \CA \ &  
}
\end{aligned}
\end{equation}
subject to the conditions that the components
of the $\lambda_n$ at $I$ are all equal and there are the 
equations \eqref{alphacond} and \eqref{omegalambdacond} of pasting diagrams.} 
\end{definition}

\begin{eqnarray}\label{alphacond}
\begin{aligned}
\xymatrix{
\CA^{\omega}\ar[r]^{\zeta_*}_{~}="1" \ar[ddr]_{\otb}^(0.35){~}="4" & \CA^{\omega}\ar[r]^{\xi_*}_{~}="2" \ar[dd]_{\otb}_(0.3){~}="6a"^(0.33){~}="6" & \CA^{\omega} \ar[ddl]^{\otb}_(0.35){~}="7"  &  & \CA^{\omega}\ar[rr]^{(\xi\zeta)_*}_{~}="3" \ar[ddr]_{\otb}^{~}="5" & & \CA^{\omega}\ar[ddl]^{\otb}_{~}="8" \\
& & & = \\
& \CA  & & & & \CA  &
\ar@{=>}"4";"6a"^{\alpha_{\zeta}}
\ar@{=>}"6";"7"^{\alpha_{\xi}}
\ar@{=>}"5";"8"^{\alpha_{\xi\zeta}} 
}
\end{aligned}
\end{eqnarray}

\begin{eqnarray}\label{omegalambdacond}
\begin{aligned}
\xymatrix{
\CA\ar[r]^{\delta_n}_{~}="1" \ar[ddr]_{1_{\CA}}^(0.35){~}="4" & \CA^{\omega}\ar[r]^{\xi_*}_{~}="2" \ar[dd]_{\otb}_(0.3){~}="6a"^(0.33){~}="6" & \CA^{\omega} \ar[ddl]^{\otb}_(0.35){~}="7"  &  & \CA\ar[rr]^{\delta_{\xi(n)}}_{~}="3" \ar[ddr]_{1_{\CA}}^{~}="5" & & \CA^{\omega}\ar[ddl]^{\otb}_{~}="8" \\
& & & = \\
& \CA  & & & & \CA  &
\ar@{=>}"4";"6a"^{\lambda_n}
\ar@{=>}"6";"7"^{\alpha_{\xi}}
\ar@{=>}"5";"8"^{\lambda_{\xi(n)}} 
}
\end{aligned}
\end{eqnarray}

\begin{example} {\em Every series monoidal category is $\omega$-monoidal with $\otb=\Sigma$ and $I=0$.} 
\end{example} 
 
\begin{example}\label{catmultP} {\em As a categorical version of Example~\ref{multP}, if $\CA$ is a
``series rig category'', possibly without unit, then $\CA$ becomes $\omega$-monoidal with 
$\otb$ taken to be a categorical version of the $\mathrm{P}$ of \eqref{fP} and $I=0$.} 
\end{example}

\begin{center}
--------------------------------------------------------
\end{center}

\appendix

\end{document}